\documentclass[proceedings]{dmtcs}

\usepackage[latin1]{inputenc}
\usepackage{subfigure}

\title{Matrix Ansatz, lattice paths and rook placements} 

\author{S. Corteel\addressmark{1} \and
M. Josuat-Verg\`es\addressmark{1}\thanks{Partially supported by an
Erwin Schr\"odinger fellowship and the ANR Jeune Chercheur IComb}
\and T. Prellberg \addressmark{2}\and M. Rubey\thanks{M. Rubey acknowledges
  partial support by the Austrian FWF-National Research Network
 Analytic Combinatorics and Probabilistic Number Theory,  project S9607.} \addressmark{3}}

\address{
\addressmark{1}LRI, CNRS and Universit\'e Paris-Sud, B\^atiment 490, 91405 Orsay, France

\addressmark{2}	School of Mathematical Sciences, Queen Mary, University of London, 
Mile End Road, London E1 4NS, United Kingdom

\addressmark{3}Institut f\"ur Algebra, Zahlentheorie und Diskrete Mathematik,
Leibniz Universit\"at Hannover, Welfengarten 1, 30167 Hannover, Deutschland
}

\newcommand{\qbin}[2]{\genfrac{[}{]}{0pt}{}{#1}{#2}_q}

\newcommand{\Le}{\rotatebox[origin=C]{180}{$\Gamma$}}

\newcommand{\ket}[1]{\ensuremath{|#1\rangle}}
\newcommand{\bra}[1]{\ensuremath{\langle #1|}}
\usepackage{amsmath,amsfonts,amssymb,latexsym}
\usepackage{hyperref}
\usepackage{graphicx}
\usepackage{pstricks}

\newtheorem{prop}{Proposition}
\newtheorem{lem}{Lemma}
\newtheorem{definition}{Definition}
\newtheorem{theorem}{Theorem}

\newtheorem{corollary}{Corollary}

\keywords{Enumeration, Permutation tableaux, Rook placements, Lattice paths}

\begin{document}

\maketitle

\begin{abstract}
  We give two combinatorial interpretations of the Matrix Ansatz of the
  PASEP in terms of lattice paths and rook placements.  This gives two
  (mostly) combinatorial proofs of a new enumeration formula for the
  partition function of the PASEP.  Besides other interpretations,
  this formula gives the generating function for permutations of a given
  size with respect to the number of ascents and occurrences of the
  pattern 13-2, the generating function according to weak exceedances 
  and crossings, and the $n$\textsuperscript{th} moment of certain
  $q$-Laguerre polynomials.

\noindent {\bf R\'esum\'e.}

Nous donnons deux interpr\'etations combinatoires du Matrix Ansatz du
PASEP en termes de chemins et de placements de tours. Cela donne deux
preuves (presque) combinatoires d'une nouvelle formule pour la
fonction de partition du PASEP. Cette formule donne aussi par exemple
la fonction g\'en\'eratrice des permutations de taille donn\'ee par rapport
au nombre de mont\'ees et d'occurrences du motif 13-2, la
fonction g\'en\'eratrice par rapport au nombre d'\'exc\'edences faibles et de
croisements, et le $n$\textsuperscript{i\`eme} moment de certains
polyn\^omes de $q$-Laguerre.

\end{abstract}

\section{Introduction}

In recent work of Postnikov \cite{AP}, permutations were given a new
description as pattern-avoiding fillings of Young diagrams.  More
precisely, Postnikov made a correspondence between positive Grassmann cells,
pattern-avoiding fillings called \Le-diagrams, and decorated
permutations (which are permutations where the fixed points are
bi-coloured). In particular, the usual permutations are in bijection with
permutation tableaux, a subclass of \Le-diagrams.  Permutation
tableaux have subsequently been studied by Steingr\`imsson, Williams,
Burstein, Corteel, Nadeau \cite{Bu,CN,CW,SW}, and proved to be very useful
for working on permutations.


Rather surprisingly, Corteel and Williams established a link between
permutation tableaux and the stationary distribution of a
classical process studied in statistical physics, the Partially
Asymmetric Exclusion Process (PASEP).  This process is described
in \cite{CW,DEHP}. Briefly, the stationary probability of a given
state in the process is proportional to the sum of weights of
permutation tableaux of a given shape.  The factor behind this
proportionality is the partition function, which is the sum of weights
of permutation tableaux of a given half-perimeter.


An alternative way of finding the stationary distribution of the PASEP is given
by the Matrix Ansatz \cite{DEHP}. Suppose that we have operators $D$ and $E$,
a row vector $\bra{W}$ and a column vector $\ket{V}$ such that:
\begin{equation} 
  DE-qED = D+E, \qquad \bra{W}E = \bra{W}, \qquad D\ket{V} = \ket{V}, \qquad \hbox{and}
  \quad \bra{W}\ket{V}=1. \label{ansatz} 
\end{equation}
Then, coding any state of the process by a word $w$ of length $n$ in
$D$ and $E$, the probability of the state $w$ is given by
$\bra{W}w\ket{V}$ normalised by the partition function $\bra{W}(D+E)^n\ket{V}$.


We briefly describe how the Matrix Ansatz is related to permutation
tableaux \cite{CW}.  First, notice that there are unique polynomials
$n_{i,j}\in \mathbb{Z}[q]$ such that
$$ (D+E)^n = \sum_{i,j\geq 0} n_{i,j} E^i D^j $$
This sum is called the normal form of $(D+E)^n$. It is useful since,
for example, the sum of coefficients $n_{i,j}$ gives an evaluation of
$\bra{W}(D+E)^n\ket{V}$.  Each coefficient $n_{i,j}$ is a generating
function for permutation tableaux satisfying certain conditions, or
equivalently, {\it alternative tableaux} as defined by Viennot
\cite{XGV}.

We give here two combinatorial interpretations of the Matrix Ansatz
in terms in lattice paths and rook placements, and get two
semi-combinatorial proofs of the following theorem:
\begin{theorem} \label{main} For any $n>0$, we have: 
$$
\bra{W} (yD+E)^{n-1} \ket{V} = \tfrac 1{y(1-q)^n} \sum_{k=0}^n (-1)^k 
\left(\sum_{j=0}^{n-k} y^j\Big( \tbinom{n}{j}\tbinom{n}{j+k} - 
   \tbinom{n}{j-1}\tbinom{n}{j+k+1}\Big)\right)
\left(\sum_{i=0}^k y^iq^{i(k+1-i)}  \right).
$$
\end{theorem}

The combinatorial interpretation of this polynomial, in terms of
permutations, is given in Proposition~\ref{comb}. For $y=1$ this
specialises to:

\begin{corollary}\label{main2} For any $n>0$, we have: 
$$
\bra{W}(D+E)^{n-1}\ket{V} = \frac 1{(1-q)^n} \sum_{k=0}^n (-1)^k 
\Big( \tbinom{2n}{n-k} - \tbinom{2n}{n-k-2} \Big)
\Bigg( \sum_{i=0}^k q^{i(k+1-i)} \Bigg). 
\label{cr}
$$
\end{corollary}


Besides the references mentioned earlier, we have to point out an
article of Williams \cite{LW}, where we find the following formula for
the coefficient of $y^{m-1}$ in $\bra{W}(yD+E)^n\ket{V}$:
\begin{equation}
E_{m,n}(q)=\sum_{i=0}^{m-1}(-1)^i[m-i]_q^nq^{mi-m^2}
\left( \tbinom n i q^{m-i} + \tbinom n{i-1}
\right).
\label{lauren}
\end{equation}
It was obtained by enumerating \Le-diagrams of a given shape and then
computing the sum of all possible shapes.  Until now it was the only
known polynomial formula for the distribution of a permutation pattern
of length greater than two (See Proposition \ref{comb}).  Although the
article \cite{LW} focuses on \Le-diagrams, Williams and her coauthors
sketched in Section~4 of \cite{NTW} how this could have been done
directly on permutation tableaux.  Recently, Williams's formula has
been obtained also by Kasraoui, Stanton and Zeng in their work on
orthogonal polynomials \cite{KSZ}.  We will show in the last section
how our formula can be applied to prove and extend a conjecture
presented in \cite{LW}.

The polynomial $y\bra{W}(yD+E)^{n-1}\ket{V}$ was already heavily
studied.
\begin{prop}   \label{comb}
For any $n\geq 1$ the following polynomials are equal:
\begin{itemize}
\item $y\bra{W}(yD+E)^{n-1}\ket{V}$,
\item the generating function for permutation tableaux of size $n$,
  the number of lines counted by $y$ and the number of superfluous 1's
  counted by $q$ \cite{CW,XGV2},
\item the generating function for permutations of size $n$, the number
  of ascents counted by $y$ and the number of 13-2 patterns counted by
  $q$ \cite{CN,SW},
\item the generating function for permutations of size $n$, the number
  of weak exceedances counted by $y$ and the number of crossings
  counted by $q$ \cite{Co,SW},
\item the generating function of PDSAWs (partially directed self-avoiding 
  walks) in the asymmetric wedge of length $n$ where the
  number of descents is counted by $y$ and the number of north steps
  is counted by $q$ \cite{R},
\item the $n$\textsuperscript{th} moment of the $q$-Laguerre polynomials
  \cite{KSZ, R}.
\end{itemize}
\end{prop}

\noindent {\bf Remark.} We can view the formula in
Corollary~\ref{main2} as an analog of the Touchard-Riordan formula
\cite{JT} for the number of matchings of $2n$ according to the number
of crossings:
$$ \sum_{\text{$M$ matching of $2n$}} q^{\hbox{cr}(M)} = \frac 1 {(1-q)^n}
\sum_{k=0}^n (-1)^k \left( \binom{2n}{n-k} - \binom{2n}{n-k-1} \right)
q^{\frac {k(k+1)}2}.
$$
We remark that this formula also gives the $2n$\textsuperscript{th}
moment of the $q$-Hermite polynomials.

In \cite{JGP}, Penaud gave a combinatorial proof of this formula.  By
generalising Penaud's method we conjectured Theorem \ref{main} and
were hoping for a completely combinatorial proof thereof.  However, at
the time of writing the last step of this combinatorial proof is still
missing.


This article is organised as follows: we first show how the Matrix
Ansatz is naturally related to lattice paths.  Then we give two proofs
of our main Theorem, one based on lattice paths and the other one
based on rook placements.  We end with a discussion and some
applications.



\section{A first proof using lattice paths and functional equations}

\subsection{The Matrix Ansatz and lattice paths}

We follow the ideas developed in \cite{BCEPR,BE}.  Looking for a
solution of the system defined in Equation~\eqref{ansatz} we find:
\begin{prop}\label{prop:solution}
  Let $D=(D_{i,j})_{i,j\ge 0}$ and $E=(E_{i,j})_{i,j\ge 0}$ such that
\begin{eqnarray*}
  D_{i,j}&=&
  \left\{\begin{array}{ll}
      1+\ldots +q^{i} &\text{if $i$ equals $j-1$  or $j$,}\\
      0 & \text{otherwise,}
    \end{array} \right.\\
  E_{i,j}&=&
  \left\{\begin{array}{ll}
      1+\ldots +q^{i} &\text{if $i$ equals $j$  or $j+1$,}\\
      0 & {\rm otherwise,}
    \end{array} \right.\\
  \bra{W}&=&(1,0,0,\ldots)\text{, and}\\
  \ket{V}&=&(1,0,0,\ldots)^T\;.
\end{eqnarray*}
Then these matrices and vectors satisfy the Ansatz of
Equation~\eqref{ansatz}.
\end{prop}

We can interpret $y\bra{W}(yD+E)^{n-1}\ket{V}$ as the
generating polynomial of paths of length $n-1$. The weight of a path
is the product of the weight of its steps and the weight of the
starting and ending points.  If a path starts (resp. ends) at $(0,i)$
(resp. $(n-1,i)$) the weight of the starting (resp. ending) point is
$W_i$ (resp. $V_i$). The weight of a step going from $(x,i)$ to
$(x+1,j)$ is $D_{i,j}+E_{i,j}$.  We call $i$ the starting height of
the step.  See \cite{BCEPR,BE} for details.

Proposition~\ref{prop:solution} implies that the paths we are dealing
with here are bi-coloured Motzkin paths, i.e., paths that start and end
at height zero and consist of north-east, south-east and two types of
east steps.  Using a classical bijection we can transform these paths
of length $n-1$ into Motzkin paths of length $n$ where east steps
of type~2 can not appear at height zero.

\begin{prop}\label{prop:histoires-laguerre}
  $y\bra{W}(yD+E)^{n-1}\ket{V}$ is the generating polynomial of
  weighted bi-coloured Motzkin paths of length $n$ such that the weight
  of steps starting at height $i$ is
  \begin{itemize}
  \item $y+yq+\ldots+yq^i=y\frac{1-q^{i+1}}{1-q}$ for north-east
    steps and east steps of type~1, and
  \item $1+q+\ldots+q^{i-1}=\frac{1-q^i}{1-q}$ for south-east steps
    and east steps of type~2.
\end{itemize}
\end{prop}
This can also be done combining results in \cite{Co,CW,SW}.

\subsection{The proof}

The method used in this subsection is inspired by an article of Penaud
\cite{JGP}.  We extract a factor of $(1-q)^n$ from the generating
polynomial of the weighted bi-coloured Motzkin paths from
Proposition~\ref{prop:histoires-laguerre} and obtain that
$$
y\bra{W}(yD+E)^{n-1}\ket{V}= \frac{1}{(1-q)^n}\sum_{p\in P(n)}w(p),
$$ 
where $P(n)$ is the set of labelled bi-coloured Motzkin paths of length
$n$ such that the weight of steps starting at height $i$ is either

\begin{itemize}
\item $y$ or $-yq^{i+1}$ for north-east steps or east steps of type~1.
\item $1$ or $-q^i$ for south-east steps or east steps of type~2.
\end{itemize}

Let $M(n)$ be the subset of the paths in $P(n)$ such that the weight
of any east step and the weight of any peak (a north-east step
followed by a south-east step) is neither $1$ nor $y$.  Let ${\mathcal
  M}_{n,k,j}$ be the number of left factors of bi-coloured Motzkin
paths of length $n$, final height $k$, and with $j$ south-east and
east steps of type~1.
\begin{lem}
  There is a bijection between paths in $P(n)$ and pairs of paths such
  that
  \begin{itemize}
  \item the first path is a left factor of a bi-coloured Motzkin path
    of length $n$ and final height $k$
  \item the second path is in $M(k)$, for some $k$ between $0$ and
    $n$.
  \end{itemize}
  In particular, we have
  $$
  \sum_{p\in P(n)}w(p)= 
  \sum_{k=0}^n\sum_{j=0}^{n-k}
  {\mathcal M}_{n,k,j}y^j 
  \sum_{p\in M(k)}w(k).
  $$
\end{lem}

%
\begin{proof} Let $p$ be a path in $P(n)$.  We decompose $p$ into a
  sequence $m_1 q_1 m_2 q_2 \dots m_k q_k m_{k+1}$ such that
  \begin{itemize}
  \item the $m_i$ are maximal (but possibly empty) sub-paths of $p$
    with all steps having weight $1$ or $y$, and returning to their
    starting height,
  \item the $q_i$ are single steps.
  \end{itemize}
  It follows that $q_1 q_2\dots q_k$ is a path in $M(k)$.  Replacing
  in the sequence $m_1 q_1 m_2 q_2 \dots m_k q_k
  m_{k+1}$ each step $q_i$ by a north-east step, and taking into account the number of
  south-east steps and east steps of type~1, we obtain a path in
  ${\mathcal M}_{n,k,j}$ of weight $y^j$.
\end{proof}

It remains to compute ${\mathcal M}_{n,k,j}$ and $M_k=\sum_{p\in
  M(k)}w(k)$.
\begin{prop}
  The number ${\mathcal M}_{n,k,j}$ of left factors of bi-coloured
  Motzkin paths of length $n$, final height $k$, and with $j$
  south-east steps and east steps of type~1, is
  $\binom{n}{j}\binom{n}{j+k}-\binom{n}{j-1}\binom{n}{j+k+1}$.
\end{prop}

\begin{proof}
  We note that the formula can be seen as a $2\times 2$ determinant.
  By the Lindstr\"om-Gessel-Viennot lemma, this equals the number of
  pairs of non-intersecting lattice paths taking north and east steps
  from $(1,0)$ to $(n-j,j)$ and $(0,1)$ to $(n-j-k,j+k)$ respectively.

  We transform such a pair of paths step by step into a single Motzkin
  path according to the following translation table:

  \begin{center} 
    \begin{tabular}{l|l|l|l} 
      $i$\textsuperscript{th} step of & lower path & upper path & Motzkin path \\
      \hline
      & north      & north      & east type~1\\
      & east       & east       & east type~2\\
      & north      & east       & north-east\\
      & east       & north      & south-east.
    \end{tabular} 
  \end{center}

  It is easy to see that the condition that the two lattice paths do
  not intersect corresponds to the condition that the Motzkin path
  does not run below the $x$-axis.  Furthermore, we see that the
  number of east and south-east steps equals $j$, the number of north
  steps of the lower path.
\end{proof}

\begin{prop}
  The generating polynomial $M_k$ equals $\sum_{i=0}^k y^i
  q^{i(k+1-i)}$.
\end{prop}
\begin{proof} 
  We add an extra parameter on the paths in $M(n)$, that marks the
  number of steps that have a weight different from $1$ and $y$.  More
  precisely, the weight of steps starting at height $i$ is
  \begin{itemize}
  \item $y$ or $-yzq^{i+1}$ for north-east steps or east steps of
    type~1, and
  \item $1$ or $-zq^i$ for south-east steps or east steps of type~2.
  \end{itemize}

  Let $M(z)=\sum_{n\ge 0}t^n \sum_{p\in M(n)}w(p)$.  We can obtain a
  functional equation for $M(z)$ by considering the following
  decomposition: A path is either (a) empty, (b) a north-east step of weight $1$, 
  followed by a path, followed by a south east step of weight $-qyz$, followed by
  another path, (c) a north-east step of weight $-qz$, followed by a path,
  followed by a south east step of weight $-qyz$, followed by another path,
  (d) a north-east step of weight $-qz$, followed by a path,
  followed by a south east step of weight $y$, followed by another
  path, (e) a north-east step of weight $1$, followed by a
  \emph{non-empty} path, followed by a south east step of weight
  $y$, followed by another path, (f) an east step of type~1 followed by another path,
  or (g) an east step of type~2 followed by a path. The corresponding weight is  
  (a) $1$, (b) $-M(qz)qyzM(z)t^2$, (c) $qzM(qz)qyzM(z)t^2$, (d) $-qzM(qz)yM(z)t^2$,
  (e) $\left(M(qz)-1\right)yM(z)t^2$, (f) $-qyzM(z)t$, or (g)
  $-zM(z)t$, respectively.
  Thus, we have:
  $$
  M(z)=1-(qyzt+zt+yt^2)M(z)+yt^2(1-qz)^2M(z)M(qz)\;.
  $$
  Proceeding similar to \cite{PB}, we use the linearising Ansatz
  $$
  M(z)=\frac1{1-z}\frac{H(qz)}{H(z)}
  $$
  to obtain
  $$
  H(z)-(1+yt^2)H(qz)+yt^2H(q^2z)=z\left(H(z)+(1+qy)tH(qz)+qyt^2H(q^2z)\right).
  $$
  Solving recursively for the coefficients $c_n$ of $H(z)=\sum_{n=0}^\infty c_nz^n$,
  we obtain a solution in terms of a basic hypergeometric series,
  $$
  H(z)= {}_2\phi_1(-t,-tqy;t^2qy;q,z)
  =\sum_{n=0}^\infty\frac{(-t,-tqy;q)_n}{(t^2qy,q;q)_n}z^n\;.
  $$
  Note that we are dealing with a series of the type
  ${}_2\phi_1(a,b;ab;q,z)$ where $a=-t$ and $b=-tqy$. In order to take
  the limit $z\to1$, we need to transform using Heine's transformation
  $$
  {}_2\phi_1(a,b,ab;q,z)=\frac{(az,b;q)_\infty}{(ab,z;q)_\infty}{}_2\phi_1(a,z
  ;az;q,b)\;.
  $$
  We find that
  $$
  M(z)=\frac1{1-az}\frac{{}_2\phi_1(a,qz;aqz;q,b)}{{}_2\phi_1(a,z,az;q,b)}
  $$
  and therefore
  $$
  M(1)=\frac1{1-a}{}_2\phi_1(a,q;aq;q,b)=\sum_{n=0}^\infty\frac{b^n}{1-aq^n}\;
  .
  $$
  Changing back to $a=-t$ and $b=-tqy$,
  $$
  M_k=(-1)^k[t^k]M(1)=\sum_{m+n=k} y^nq^{n(m+1)}=\sum_{i=0}^ky^iq^{i(k-i+1)}.
  $$
\end{proof}

Combining the previous results, we get a proof of Theorem~\ref{main}.

\section{A second proof using the Matrix Ansatz and rook placements}

For further details about material in this section, see \cite{J}.  One
of the ideas at the origin of this proof is the following. From $D$
and $E$ of the Matrix Ansatz, we define new operators $\hat D$ and $\hat E$ as
$$ \hat D = \frac{q-1}q D + \frac 1q, \qquad \hat E = \frac{q-1}q E + \frac 1q.$$
An immediate consequence is that
\begin{equation} \label{comrel} \hat D \hat E - q \hat E \hat D =
  \frac{1-q}{q^2}, \qquad \bra{W} \hat E = \bra{W}, \qquad
  \hbox{and}\quad \hat D\ket{V} = \ket{V}.
\end{equation}
This commutation relation is somewhat simpler than the
one satisfied by $D$ and $E$, as it has no terms linear in $\hat D$ or $\hat E$. 
Moreover, we have $q(y\hat D + \hat E) + (1-q)(yD+E) =
1+y$, for any parameter $y$. Using this identity, we obtain
the following inversion formulae between $(yD+E)^n$ and $(y\hat D+\hat
E)^n$:

\begin{equation} (1-q)^n (yD+E)^n = \sum_{k=0}^n \binom n k
  (1+y)^{n-k} (-1)^k q^k(y\hat D+ \hat E)^k,\quad\hbox{and}
\label{inv1}
\end{equation}
\begin{equation}
q^n (y\hat D+\hat E)^n = \sum_{k=0}^n \binom n k (1+y)^{n-k} (-1)^k  (1-q)^k
(D+E)^k.
\label{inv2}
\end{equation}

In particular, the first formula means that in order to compute the
coefficients of the normal form of $(yD+E)^n$, it is sufficient to compute
the ones of $(y\hat D + \hat E)^k$ for all $0\leq k\leq n$ (as taking the 
normal form is a linear operation).


Except for a $q$-dependent factor, the operators $\hat D$ and
$\hat E$ are also defined in \cite{USW} and \cite{BECE}.  In the first
reference, Uchiyama, Sasamoto and Wadati used the commutation relation between
$\hat D$ and $\hat E$ to find explicit matrices for these
operators. They derive the eigenvalues and eigenvectors of $\hat
D+\hat E$, and consequently the ones of $D+E$, in terms of orthogonal
polynomials.  In the second reference, Blythe, Evans, Colaiori and
Essler also use these eigenvalues and obtain an integral form for
$\bra{W}(D+E)^n\ket{V}$. They also provide an exact integral-free
formula of this quantity, although quite complicated since it contains
three summations and several $q$-binomial coefficients.


In this article, instead of working on representations of $\hat D$ and
$\hat E$ and their eigenvalues, we study the combinatorics of the
rewriting in the normal form of $(\hat D+\hat E)^n$, and more generally
$(y\hat D+\hat E)^n$ for some parameter $y$.  In the case of $\hat D$
and $\hat E$, the objects that appear are the {\it rook placements in
  Young diagrams}, long-known by combinatorists since the results of
Kaplansky, Riordan, Goldman, Foata and Sch\"utzenberger (see
\cite{RPS} and references therein). This method is described in
\cite{AV}, and is the same that the one leading to permutation
tableaux or alternative tableaux in the case of $D$ and $E$.


\begin{definition} Let $\lambda$ be a Young diagram. A {\it rook
    placement} of shape $\lambda$ is a partial filling of the cells of
  $\lambda$ with rooks (denoted by a circle $\circ$), such that there
  is at most one rook per row (resp. per column).
\end{definition}


For convenience, we distinguish with a cross ($\times$) each cell of
the Young diagram that is not below (in the same column) or to the
left (in the same row) of a rook (we are using the French
convention). The number of crosses is an important statistic on rook
placements, which was introduced in \cite{GR} as a generalisation of
the inversion number for permutations. Indeed, if $\lambda$ is a
square of side length $n$, a rook placement $R$ with $n$ rooks may be
visualised as the graph of a permutation $\sigma\in\mathfrak{S}_n$, and in
this interpretation the number of crosses in $R$ is the inversion number of
$\sigma$.


\begin{definition}
  The weight of a rook placement $R$ with $r$ rooks, $s$ crosses and
  $t$ columns is $w(R) = p^r q^sy^t$, where $p=\frac{1-q}{q^2}$.
\end{definition}


With the definition of rook placements and their weights we can give
the combinatorial interpretation of $\bra{W}(y\hat D + \hat
E)^n\ket{V}$.

\begin{prop} For any $n$, $\bra{W}(y\hat D + \hat E)^n\ket{V}$ is
  equal to the sum of weights of all rook placements of half-perimeter
  $n$.
\end{prop}

The enumeration of rook placements leads to an evaluation of $\bra{W}
(y\hat D+\hat E)^{n-1} \ket{V}$, hence of $\bra{W} (yD+E)^{n-1}
\ket{V}$ via the inversion formula~\eqref{inv1}.

\subsection{Rook placements and involutions}

Given a rook placement $R$ of half-perimeter $n$, we define an
involution $\alpha(R)$ by the following construction: label the
north-east boundary of $R$ with integers from 1 to $n$. This implies that
each column or row has a label between 1 and $n$. If a column, or row, is
labelled by $i$ and does not contain a rook, it is a fixed point of
$\alpha(R)$. Also, if there is a rook at the intersection of column $i$
and row $j$, then $\alpha(R)$ sends $i$ to $j$ (and $j$ to $i$).


Given a rook placement $R$ of half-perimeter $n$, we also define a
Young diagram $\beta(R)$ by the following construction: if we remove
all rows and columns of $R$ containing a rook, the remaining cells
form a Young diagram, which we denote by $\beta(R)$. We also define
$\phi(R) = (\alpha(R),\beta(R))$. See Figure~\ref{phi} for an example.

\begin{figure}[htp]\begin{center}
\psset{unit=3mm}
\begin{pspicture}(-1,0)(5,5) 
\rput(-1.2,2){$R=$}
\psline(0,0)(5,0)  \psline(0,0)(0,5)
\psline(0,1)(5,1)  \psline(1,0)(1,5)
\psline(0,2)(5,2)  \psline(2,0)(2,5)
\psline(0,3)(5,3)  \psline(3,0)(3,4)
\psline(0,4)(4,4)  \psline(4,0)(4,4)
\psline(0,5)(2,5)  \psline(5,0)(5,3)
\rput(0.5,4.5){\small $\times$} \rput(1.5,4.5){\small $\times$}
\rput(0.5,3.5){\small $\circ$}  \rput(1.5,3.5){\small $\times$}
\rput(2.5,3.5){\small $\times$} \rput(3.5,3.5){\small $\times$}
\rput(3.5,2.5){\small $\circ$}  \rput(4.5,2.5){\small $\times$}
\rput(2.5,1.5){\small $\circ$}  \rput(4.5,1.5){\small $\times$}
\rput(1.5,0.5){\small $\times$}  \rput(4.5,0.5){\small $\times$}
\end{pspicture}
\hspace{1.8cm}
\begin{pspicture}(-1.5,-1)(10,2.5)
\rput(-1.5,1){$\phi(R)=\Bigg($}
\psdots(1,0)(2,0)(3,0)(4,0)(5,0)(6,0)(7,0)(8,0)(9,0)(10,0) 
\psarc(3.5,0){2.5}{0}{180}
\psarc(3.5,0){2.5}{0}{180}
\psarc(6.5,0){2.5}{0}{180}
\psarc(6.5,0){2.5}{0}{180}
\psarc(6.5,0){1.5}{0}{180} 
\psarc(6.5,0){1.5}{0}{180}
\rput(10.7,0){,}
\end{pspicture}
\hspace{0.4cm}
\begin{pspicture}(0,-1)(2.2,2)
\psline(0,0)(2,0)\psline(0,0)(0,2)
\psline(0,1)(2,1)\psline(1,0)(1,2)
\psline(0,2)(1,2)\psline(2,0)(2,1)
\rput(2.4,1){$\Bigg)$}
\end{pspicture}
\caption{Example of a rook placement and its image by the map $\phi$\label{phi}.}
\end{center}\end{figure}
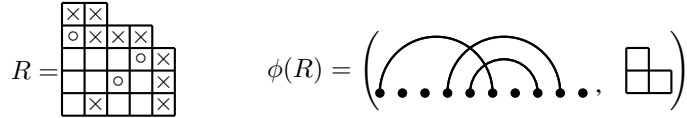

\begin{prop} The map $\phi$ is a bijection between rooks placements in
  Young diagrams of half-perimeter $n$, and ordered pairs $(I,\lambda)$
  where $I$ is an involution on $\{1,\dots,n\}$ and $\lambda$ a Young
  diagram of half-perimeter $|$Fix$(I)|$. If $\phi(R) = (I,\lambda)$,
  the number of crosses in $R$ is the sum of $|\lambda|$ and some
  parameter $\mu(I)$.
\end{prop}

\begin{proof} This kind of bijection rather classical, see for
  instance \cite{Bu,Ke}.  Note that the pairs $(I,\lambda)$ may be
  seen as involutions on $\{1,\dots,n\}$ with a weight 2 on each fixed
  point. For the second part of the proposition, we just have to
  distinguish different kinds of crosses in the rook placement
  $R$. For example, the crosses with no rook in the same line and column are
  enumerated by $|\lambda|$.
\end{proof}

\begin{corollary} Let $T_{j,k,n}$ be the sum of weights of rook
  placements of half perimeter $n$, with $k$ lines and $j$ lines
  without rooks. Then for any j,k,n, we have:
\begin{equation}
 T_{j,k,n} = \qbin{n-2k+2j}{j}y^j T_{0,k-j,n}. \label{fact}
\end{equation}
\end{corollary}

\begin{proof} The previous proposition means that the number of
  crosses is an additive parameter with respect to the decomposition
  $R\mapsto (I,\lambda)$. This naturally lead to a factorisation of
  the generating function.
\end{proof}

\subsection{The recurrence}

\begin{prop} We have the following recurrence relation:
\begin{equation}
  T_{0,k,n} = T_{0,k,n-1} + py [n+1-2k]_qT_{0,k-1,n-1}. \label{T0rec}
\end{equation}
\end{prop}

\begin{proof} We have the relation:
\begin{equation}
  T_{0,k,n} = T_{0,k,n-1} + p T_{1,k,n-1}. 
\end{equation}
Indeed, we can distinguish two cases, whether a rook placement
enumerated by $ T_{0,k,n}$ has a rook in its first column or
not. These two cases give respectively the two terms of the previous
identity. To end the proof we can apply identity (\ref{fact}) to the
second term.
\end{proof}

The recurrence~\eqref{T0rec} is solved by the following formula.

\begin{prop} %
\begin{equation}
  T_{0,k,n}  = q^{-2k} \sum_{i=0}^k (-1)^i q^{\frac {i(i+1)} 2} 
  {n-2k+i \brack i}_q \left( \binom n{k-i} - \binom n{k-i-1} \right). \label{T0kn}
\end{equation}
\end{prop}

From this proposition, identity~\eqref{fact}, and a $q$-binomial identity, we derive a formula for
$T_{j,k,n}$.

\begin{prop} 
$$
\sum_{j=0}^k T_{j,k,n} = \sum_{j=0}^k  \left(\tbinom nj - \tbinom n{j-1}\right)
\left(  \tfrac{ q^{(k+1-j)(n-k-j)} - q^{(k-j)(n-k-j)} 
 + q^{(k-j)(n+1-k-j)} - q^{(k+1-j)(n+1-k-j)} }{(1-q)q^n} \right).
\label{Tkn}
$$
\end{prop}

Summing this identity over $k$ gives the following result.

\begin{prop} \label{propTn} 
\begin{equation}  \bra{W}(y\hat D+\hat E)\ket{V} = (1+y) G(n) - G(n+1), 
\label{Tn}
\end{equation}  
$$\hbox{where \quad } G(n)=\sum_{j=0}^{\lfloor \frac n2 \rfloor} 
 \left( \binom nj - \binom n{j-1} \right) \sum_{i=0}^{n-2j} y^{i+j-1} q^{i(n+1-2j-i)}.$$
\end{prop}

This formula is a linear combination of the polynomials $P_k =
\sum_{i=0}^k y^i q^{i(k+1-i)}$, the coefficients being polynomials in
$y$, just as in Theorem~\ref{main}.  With this result and
the inversion formula~\eqref{inv1}, we can prove Theorem~\ref{main}:
the last step is a simple binomial simplification.

\section{Applications}

Among all the objects of the list in Proposition~\ref{comb}, the
most studied are probably permutations and the pattern 13-2, see for
example \cite{CM,CN,SW,RP}. In particular, in \cite{CM,RP} we can find
methods for obtaining, as a function of $n$ for a given $k$, the
number of permutations of size $n$ with exactly $k$ occurrences of
pattern 13-2. By taking the Taylor series of~\eqref{main2}, we obtain
direct and quick proofs for these results. As an illustration
we give the formulae for $k\leq 3$ in the following proposition.

\begin{prop} The order 3 Taylor series of $\bra{W}(D+E)^{n-1}\ket{V}$
  is:
$$
\bra{W}(D+E)^{n-1}\ket{V}=C_n + \binom{2n}{n-3}q + \frac
n2\binom{2n}{n-4}q^2 + \frac{(n+1)(n+2)}{6}\binom{2n}{n-5}q^3+O(q^4),
$$
where $C_n$ is the $n$th Catalan number.
\end{prop}

More generally, a computer algebra system can provide higher order
terms, for example it takes no more than a few seconds to obtain the
following closed formula for $[q^{10}]\bra{W}(D+E)^{n-1}\ket{V}$:
\begin{eqnarray*}
\tfrac {(2n)!}{10!(n+12)!(n-8)!} \Big(
{n}^{13}+70\,{n}^{12}+2093\,{n}^{11}+
32354\,{n}^{10}+228543\,{n}^{9}-318990\,{n}^{8}\\
-17493961\,{n}^{7} -104051458\,{n}^{6}-6828164\,{n}^{5} 
+2022876520\,{n}^{4}\\
+6310831968\,{n
}^{3}+5832578304\,{n}^{2}+14397419520\,n+5748019200 \Big),
\end{eqnarray*}
which is quite an improvement compared to the methods of
\cite{RP}. In addition to exact formula, we can give asymptotic estimates, for example
for the number of permutations with a given number of occurrences of pattern 13-2.


\begin{theorem} \label{asymp} For any fixed $m\geq 0$,
$$ [q^m]\bra{W}(D+E)^{n-1}\ket{V}  \sim 
\frac{4^nn^{m-\frac 32}}{\sqrt{\pi} m!}\quad\text{as $n\to\infty$.} $$
\end{theorem}

\begin{proof} When $n\to\infty$, the numbers $\tbinom{2n}{n-k}
  - \tbinom{2n}{n-k-2}$ are dominated by the Catalan number $\frac
  1{n+1} \tbinom{2n}{n}$. This implies that in
  $(1-q)^n\bra{W}(D+E)^{n-1}\ket{V}$, each higher order term grows at
  most as fast as the constant term $C_n$. On the other side, the
  coefficient of $q^m$ in $(1-q)^{-n}$ is asymptotically $n^m / m!$. 
\end{proof}


Since any occurrence of the pattern 13-2 in a permutation is also an
occurrence of the pattern 1-3-2, a permutation with $k$ occurrences of
the pattern 1-3-2 has at most $k$ occurrences of the pattern 13-2. So
we get the following corollary.


\begin{corollary} Let $\psi_k(n)$ be the number of permutations in $\mathfrak{S}_n$
with at most $k$ occurrences of the pattern 1-3-2.
For any constant $C>1$ and $k\geq 0$, we have
$$ \psi_k(n) \leq C\frac{4^nn^{k-\frac 32}}{\sqrt{\pi} k!}$$
when $n$ is sufficiently large.
\end{corollary}


So far we have only used Corollary~\ref{main2}.  Now we illustrate
what can be done with the refined formula given in Theorem~\ref{main}.
For example, when $q=0$ then the coefficient of $y^m$ is given by the expression
$\sum_{k=0}^n(-1)^k\left( \tbinom{n}{m}\tbinom{n}{m+k} -
  \tbinom{n}{m-1}\tbinom{n}{m+k+1}\right)$. This is equal to the
Narayana number $N(n,m) = \frac 1n\tbinom{n}{m}\tbinom{n}{m-1}$ (see
\cite{LW} for a combinatorial proof).


We can also get the coefficients for higher powers of $q$. For example
it is conjectured in \cite{LW} that the coefficient of $qy^m$ in
$\bra{W}y(yD+E)^{n-1}\ket{V}$ is equal to
$\tbinom{n}{m+1}\tbinom{n}{m-2}$. Applying our results we can prove:

\begin{prop} The coefficients of $qy^m$ and $q^2y^m$ in $\bra{W}y(yD+E)^{n-1}\ket{V}$ 
are respectively:
$$ \binom{n}{m+1}\binom{n}{m-2} \qquad \hbox{ and } \qquad \binom{n+1}{m-2}
\binom{n+1}{m+2}\frac{nm+m-m^2-4}{2(n+1)}. $$
\end{prop}

\begin{proof} A naive expansion of the Taylor series in $q$ gives a
  lengthy formula, which is simplified easily after noticing that it is the
  product of $\tbinom n m ^2$ and a rational fraction of $n$ and $m$.
\end{proof}

\acknowledgements The authors want to thank Lauren Williams for her
interesting discussions and suggestions.


\begin{thebibliography}{999}
\bibitem{BECE}
R. A. Blythe, M. R. Evans, F. Colaiori and F. H. L. Essler,
Exact solution of a partially asymmetric exclusion model using a deformed 
oscillator algebra, J. Phys. A: Math. Gen. Vol. 33, (2000), 2313-2332.
\bibitem{BCEPR} R. Brak, S. Corteel, J. Essam, R. Parviainen and A. 
Rechnitzer,  
A combinatorial derivation of the PASEP stationary state.  
Electron. J. Combin.  13  (2006),  no. 1, 108, 23 pp. 
\bibitem{BE} R. Brak, and J.W. Essam,  
Asymmetric exclusion model and weighted lattice paths.  
J. Phys. A  37  (2004),  no. 14, 4183--4217.
\bibitem{Bu}
A. Burstein,  On some properties of permutation tableaux,
Ann. Combin. 11(3-4), (2007), 355-368.
\bibitem{CM}
A. Claesson, T. Mansour,
Counting Occurrences of a Pattern of Type (1,2) or (2,1) in Permutations,
Adv. in App. Maths. 29, (2002), 293-310.
\bibitem{Co}
S. Corteel, Crossings and alignments of permutations, Adv. in App. Maths.
38(2), (2007), 149-163.
\bibitem{CN}
S. Corteel and P. Nadeau, Bijections for permutation tableaux,
Eur. Jour. of Comb 30(1), (2009), 295-310.
\bibitem{CW}
S. Corteel and L. K. Williams, Tableaux combinatorics for the asymmetric 
exclusion process, Adv. in App. Maths. 39(3), (2007), 293-310.
\bibitem{DEHP}
B. Derrida, M. Evans, V. Hakim, V. Pasquier, Exact solution of a 1D asymmetric 
exclusion model using a matrix formulation, J. Phys. A: Math. Gen. 26 (1993), 
1493-1517.
\bibitem{GR}
A. Garsia and J. Remmel, q-Counting rook configurations and a formula of
Frobenius. J. Combin. Theory, Ser. A 41, (1986), 246-275.
\bibitem{J} M. Josuat-Verg\`es, Rook placements in Young diagrams and permutation enumeration, submitted (2008). arXiv:0811.0524 
\bibitem{KSZ}
A. Kasraoui, D. Stanton, J. Zeng, The combinatorics of Al-Salam-Chihara 
$q$-Laguerre polynomials, preprint (2008). arXiv:0810.3232v1.
\bibitem{Ke}
S. Kerov, Rooks on Ferrers Boards and Matrix Integrals, Zapiski. Nauchn. 
Semin. POMI, v.240 (1997), 136-146.
\bibitem{NTW} J.C. Novelli, J-Y. Thibon and L.K. Williams,
Combinatorial Hopf algebras, noncommutative Hall-Littlewood functions, and permutation tableaux, preprint (2008). arXiv:0804.0995
\bibitem{AP}
A. Postnikov, Total positivity, Grassmannians, and networks, Preprint (2006).
arxiv:math.CO/0609764. 
\bibitem{PB}
T. Prellberg and R. Brak, 
Critical Exponents from Non-Linear Functional Equations for Partially Directed Cluster Models, J. Stat. Phys., Vol. 78 (1995), 701-730.
\bibitem{RPS}
R. P. Stanley, Enumerative combinatorics Vol. 1, Cambridge university press 
(1986). 
\bibitem{SW} E. Steingr\'imsson, L. K. Williams, Permutation tableaux and
permutation patterns, J. Combin. Theory Ser. A, Vol. 114(2),
(2007), 211-234.
\bibitem{RP}
R. Parviainen, Lattice path enumeration of permutations with $k$ occurrences of the 
pattern 2-13, Journal of Integer Sequences, Vol. 9 (2006), Article 06.3.2.
\bibitem{JGP} J.-G. Penaud, A bijective proof of a Touchard-Riordan formula,
Disc. Math., Vol. 139 (1995), 347-360.
\bibitem{R} M. Rubey, Nestings of Matchings and 
Permutations and North Steps in PDSAWs, FPSAC2008 (2008). arXiv:0712.2804
\bibitem{JT}
J. Touchard, 
Sur un probl\`eme de configurations et sur les fractions continues,
Can. Jour. Math., Vol. 4 (1952), 2-25.
\bibitem{USW}
M. Uchiyama, T. Sasamoto, M. Wadati, Asymmetric simple exclusion process with 
open boundaries and Askey-Wilson polynomials, J. Phys. A: Math. Gen. 37 (2004),
4985-5002.
\bibitem{AV}
A. Varvak, Rook numbers and the normal ordering problem, 
J. Combin. Theory Ser. A,  Vol. 112(2), (2005), 292-307.
\bibitem{XGV} 
X.G. Viennot,  Alternative tableaux and permutations, in preparation (2008).
\bibitem{XGV2}
X.G. Viennot, Alternative tableaux and partially asymmetric exclusion process, in 
preparation (2008).
\bibitem{LW}
L. K. Williams,  Enumeration of totally positive Grassmann cells,
Adv. Math., Vol. 190(2), (2005), 319-342.
\end{thebibliography}
\end{document}